\documentclass[oneside,reqno,12pt]{amsart}

\usepackage{lutz-dhotc}

\title{Discrete homotopy of token configurations}
\author{Bob Lutz}
\address{Mathematical Sciences Research Institute, 17 Gauss Way, Berkeley, CA 94720}
\email{boblutz13@gmail.com}
\thanks{Work of the author was supported by NSF grant DMS-1440140 while in residence at the Mathematical Sciences Research Institute in Berkeley, California during the spring 2020 semester.}

\begin{document}
\begin{abstract}
This paper studies graphical analogs of symmetric products and unordered configuration spaces in topology. We do so from the perspective of the discrete homotopy theory introduced by Barcelo et al. Our first result is a combinatorial version of a theorem of P. A. Smith, which says that the fundamental group of any nontrivial symmetric product of $X$ is isomorphic to $H_1(X)$. Our second result gives conditions under which the $n$-strand braid group of a graph is isomorphic to its discrete analog.
\end{abstract}
\maketitle

\section{Introduction}

Interesting classes of spaces arise from the action of the symmetric group $\Sigma_n$ on the coordinates of a product space. For example, the \emph{symmetric product} $\sp{X}{n}$ of a topological space $X$ is the quotient of the Cartesian product $X^n$ by $\Sigma_n$. Symmetric products are studied in topology for their nice homotopical properties, in geometry when $X$ is an algebraic curve, and in physics as examples of orbifolds \cite{aguilar2008, bantay2003, macdonald1962}.

Another example is the \emph{(unordered) configuration space} $\uc{X}{n}$, defined as the quotient by $\Sigma_n$ of the space
\[\{(x_1,\ldots,x_n)\in X^n : x_i\neq x_j\mbox{ if }i\neq j\}.\]
Configuration spaces and their fundamental groups, called \emph{braid groups}, are central objects in many areas, including knot theory, mapping class groups and motion planning \cite{birman1974, farber2003}.

In this paper, we study combinatorial analogs of these quotient spaces. Given a graph $G$, let $\rp{G}{n}$ denote the quotient graph of the Cartesian graph product $G^n$ by the action of $\Sigma_n$. This quotient is sometimes called a \emph{reduced power} of $G$ \cite{hammack2016}. The vertices of $\rp{G}{n}$ can be regarded as configurations of $n$ identical tokens on the vertices of $G$ with overlaps allowed. Let $\tok{G}{n}$ denote the subgraph of $\rp{G}{n}$ induced by all token configurations with no overlaps. This subgraph is called the \emph{$n$-token graph} of $G$ \cite{fabmon2012}.

\begin{figure}[ht]
\begin{tikzpicture}[scale=0.95]
\def\a{3.5}

\draw (0,0) circle (1);
\draw[fill=white] (0:1) circle (0.14);
\draw[fill=white] (40:1) circle (0.14);
\draw[fill=black] (80:1) circle (0.14);
\draw[fill=black] (120:1) circle (0.14);
\draw[fill=white] (160:1) circle (0.14);
\draw[fill=white] (200:1) circle (0.14);
\draw[fill=white] (240:1) circle (0.14);
\draw[fill=black] (280:1) circle (0.14);
\draw[fill=white] (320:1) circle (0.14);

\draw (\a,0) circle (1);
\draw[fill=white] ($(\a,0)+(0:1)$) circle (0.14);
\draw[fill=black] ($(\a,0)+(40:1)$) circle (0.14);
\draw[fill=white] ($(\a,0)+(80:1)$) circle (0.14);
\draw[fill=black] ($(\a,0)+(120:1)$) circle (0.14);
\draw[fill=white] ($(\a,0)+(160:1)$) circle (0.14);
\draw[fill=white] ($(\a,0)+(200:1)$) circle (0.14);
\draw[fill=white] ($(\a,0)+(240:1)$) circle (0.14);
\draw[fill=black] ($(\a,0)+(280:1)$) circle (0.14);
\draw[fill=white] ($(\a,0)+(320:1)$) circle (0.14);

\draw (2*\a,0) circle (1);
\draw[fill=white] ($(2*\a,0)+(0:1)$) circle (0.14);
\draw[fill=black] ($(2*\a,0)+(40:1)$) circle (0.14);
\draw[fill=white] ($(2*\a,0)+(80:1)$) circle (0.14);
\draw[fill=white] ($(2*\a,0)+(120:1)$) circle (0.14);
\draw[fill=black] ($(2*\a,0)+(160:1)$) circle (0.14);
\draw[fill=white] ($(2*\a,0)+(200:1)$) circle (0.14);
\draw[fill=white] ($(2*\a,0)+(240:1)$) circle (0.14);
\draw[fill=black] ($(2*\a,0)+(280:1)$) circle (0.14);
\draw[fill=white] ($(2*\a,0)+(320:1)$) circle (0.14);

\draw (3*\a,0) circle (1);
\draw[fill=white] ($(3*\a,0)+(0:1)$) circle (0.14);
\draw[fill=black] ($(3*\a,0)+(40:1)$) circle (0.14);
\draw[fill=white] ($(3*\a,0)+(80:1)$) circle (0.14);
\draw[fill=white] ($(3*\a,0)+(120:1)$) circle (0.14);
\draw[fill=black] ($(3*\a,0)+(160:1)$) circle (0.14);
\draw[fill=white] ($(3*\a,0)+(200:1)$) circle (0.14);
\draw[fill=white] ($(3*\a,0)+(240:1)$) circle (0.14);
\draw[fill=white] ($(3*\a,0)+(280:1)$) circle (0.14);
\draw[fill=black] ($(3*\a,0)+(320:1)$) circle (0.14);
\end{tikzpicture}
\caption{A path of 3-token configurations on the 9-cycle graph.}
\end{figure}
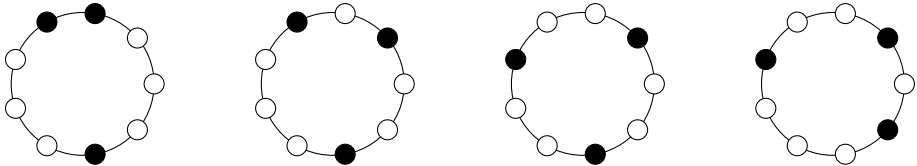

Mirroring the interest in homotopy groups of $\sp{X}{n}$ and $\uc{X}{n}$, we study the graphical analogs $\rp{G}{n}$ and $\tok{G}{n}$ in terms of the discrete homotopy theory introduced in \cite{barcelo2001}. In this theory, intervals are replaced by path graphs, and continuous maps $S^n\to X$ are replaced by graph maps $\Z^n\to G$ with finite support. The relevant groups $A_n(G)$, called the \emph{discrete homotopy groups} of $G$, are defined combinatorially. Originally used to study complex systems and their dynamics, this theory has found intriguing applications to subspace arrangements and group theory \cite{barcelo2005, barcelo2011, delabie2020}. An accompanying homology theory was introduced in \cite{barcelo2014}; the resulting groups $\hc_n(G)$ are called the \emph{discrete singular cubical homology groups} of $G$. 

A key feature of the symmetric product is that it turns homotopy into homology. More precisely, it is often the case that $\pi_k(\sp{X}{n}) \cong H_k(X)$. The earliest result in this vein is due to P. A. Smith \cite{smith1936}, who showed essentially that if $n\geq 2$, then $\pi_1(\sp{X}{n}) \cong H_1(X)$ for any CW complex $X$ (see \cite[Satz 12.15]{dold1961}). We prove a discrete version of this result:

\begin{thm}
If $n\geq 2$, then the discrete fundamental group $A_1(\rp{G}{n})$ is isomorphic to the first discrete singular cubical homology group $\hc_1(G)$.
\label{thm:symab}
\end{thm}

In essence, Theorem \ref{thm:symab} says that $\rp{G}{n}$ abelianizes the discrete fundamental group of $G$ when $n\geq 2$. To prove this, we combine the argument of Smith with a discrete Hurewicz theorem in dimension 1 \cite[Theorem 4.1]{barcelo2014} and the main result of \cite{hammack2016}, which describes an explicit cycle basis of $\rp{G}{n}$.

Our second result connects discrete homotopy theory to braid groups. Recall that the \emph{braid groups} of a space $X$ are defined as $B_n(X)=\pi_1(\uc{X}{n})$. Classically, the study of braid groups was restricted to manifolds \cite{birman1969}. For example, the Artin braid groups can be defined as $B_n(\mathbb{R}^2)$. The braid groups of graphs, regarded as 1-dimensional CW complexes, are of considerable interest among non-manifolds \cite{abrams2000, crisp2004, farley2005, ko2012}.

One way to think of points in $\uc{G}{n}$ is as configurations of robots moving continuously about a factory floor, where the edges of $G$ represent tracks or guidewires \cite{abrams2002}. The braid groups $B_n(G)$ measure the complexity of control schemes for this system \cite{ghrist2001, ghrist2007}. When $n$ is small, the discrete fundamental group $A_1(\tok{G}{n})$ provides the same data as the braid group $B_n(G)$, except that it ignores \emph{local exchanges} of robots around small cycles of $G$. Thus when $G$ contains no small cycles, the groups are the same:

\begin{thm}
If $G$ is sufficiently subdivided for $n$ and contains no 3- or 4-cycles, then the discrete fundamental group $A_1(\tok{G}{n})$ is isomorphic to the braid group $B_n(G)$.
\label{thm:tokbraid}
\end{thm}

The hypothesis that $G$ is \emph{sufficiently subdivided} simply places an upper bound on $n$; we will give a proper definition in Section \ref{sec:tok}. Our proof of Theorem \ref{thm:tokbraid} uses a cubical complex $\ud{G}{n}$ introduced by Abrams \cite{abrams2000}. This space, called the \emph{(unordered) discrete configuration space} of $G$, is a ``skeletonized'' version of $\uc{G}{n}$.

\begin{cor}
Fix a graph $G$ and a positive integer $n$. By subdividing the edges of $G$, one can obtain a graph $H$ such that $A_1(\tok{H}{n})$ is isomorphic to $B_n(G)$.
\end{cor}

The paper is organized as follows. In Section \ref{sec:dischom} we review the basics of discrete homotopy theory and discrete singular cubical homology. In Section \ref{sec:tok} we discuss basic properties of reduced powers and token graphs and provide examples. In Section \ref{sec:symab} we prove Theorem \ref{thm:symab}. In Section \ref{sec:tokbraid} we prove Theorem \ref{thm:tokbraid} and discuss the meaning of \emph{local exchanges}. Finally, in Section \ref{sec:openq} we pose several open questions.
\section{Discrete homotopy and homology groups}
\label{sec:dischom}

By a \emph{graph} we will mean a connected, simple, locally finite one. We write $u\simeq v$ when $u$ and $v$ are adjacent or equal vertices of a graph. For graphs $G$ and $H$, let us write $f:G\to H$ when $f$ is a function from the vertex set of $G$ to the vertex set of $H$. A \emph{graph map} is a function $f:G\to H$ such that if $u\simeq v$, then $f(u)\simeq f(v)$. Let $\Z$ denote the graph whose vertices are the integers, with an edge between $i$ and $j$ if and only if $|i-j|=1$. For $m\geq 0$, let $I_m$ denote the subgraph of $\Z$ induced by $\{0,\ldots,m\}$.

\subsection{Discrete fundamental group}

Fix $v_0\in G$. We write $f:(\Z,\partial\Z)\to (G,v_0)$ if $f$ is a function $\Z\to G$ and there exists an integer $r_f\geq 0$ such that $f(i)=v_0$ whenever $|i|\geq r_f$. We will assume that $r_f$ is the minimum such integer. The \emph{concatenation} of two graph maps $f,g:(\Z,\partial \Z)\to (G,v_0)$ is the graph map $p:(\Z,\partial\Z)\to (G,v_0)$ given by
\[p(i)=\begin{cases} f(i+r_f)&\mbox{if }i\leq 0\\ g(i-r_g)&\mbox{if }i\geq 0\end{cases}\]
Let $h:\Z\times I_m\to G$ be a graph map for some $m$, and write $h_j(i)=h(i,j)$ for all $i$ and $j$. We say that $h$ is a \emph{based homotopy} from $f$ to $g$ if
\begin{enumerate}
\item $h_0=f$ and $h_m=g$
\item $h_j$ is a graph map $(\Z,\partial \Z)\to (G,v_0)$ for all $j$.
\end{enumerate}
Based homotopy defines an equivalence relation on graph maps $(\Z,\partial\Z)\to (G,v_0)$.

\begin{mydef}
The \emph{discrete fundamental group} of $G$ is the set $A_1(G,v_0)$ of based homotopy classes $[f]$ of graph maps $f:(\Z,\partial \Z)\to (G,v_0)$, endowed with a group structure as follows. The identity element is the class of the identity graph map $\Z\to v_0$. The product of two classes $[f]$ and $[g]$ is the class of the concatenation of $f$ and $g$.
\end{mydef}

We will not prove that the group operation in $A_1(G,v_0)$ is well defined, or that it satisfies the group axioms. Our definition differs slightly, but not materially, from the original definition in \cite{barcelo2001}. Using similar ideas, one can define an infinite family of \emph{discrete homotopy groups} $A_n(G,v_0)$. Since we have assumed that $G$ is connected, the group $A_1(G,v_0)$ does not depend on the choice of base vertex $v_0$. We therefore write $A_1(G)=A_1(G,v_0)$.

\begin{prop}[{\cite[Proposition 5.12]{barcelo2001}}]
Let $X(G)$ denote the CW complex obtained from $G$ by attaching a 2-cell to each 3-cycle and chordless 4-cycle of $G$. We have $A_1(G)\cong \pi_1(X(G))$.
\label{prop:dfg}
\end{prop}

\begin{eg}
Consider the 2-dimensional cubical complex $X$ pictured in Figure \ref{fig:kb}. Let $G$ be the 1-skeleton of $X$. We have $X(G)=X$, which is homeomorphic to the Klein bottle, so Proposition \ref{prop:dfg} says that $A_1(G)\cong \langle a,b\mid aba=b\rangle$.

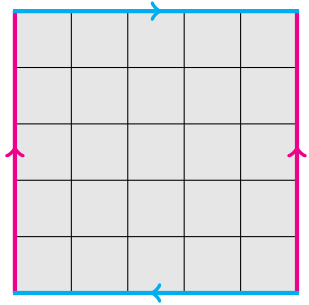
\begin{figure}[ht]
\begin{tikzpicture}[scale=0.75]
\def\a{0.035}
\draw[fill=black!10] (0,0) rectangle (5,5);
\foreach\x in {0,...,5} {
\draw (0,\x) -- (5,\x);
\draw (\x,0) -- (\x,5);
}
\draw[ultra thick,magenta,->] (0,0) -- (0,2.6);
\draw[ultra thick,magenta,->] (5,0) -- (5,2.6);
\draw[ultra thick,magenta] (0,2.5) -- (0,5);
\draw[ultra thick,magenta] (5,2.5) -- (5,5);
\draw[ultra thick,cyan,->] (-\a,5) -- (2.6,5);
\draw[ultra thick,cyan,->] (5+\a,0) -- (2.4,0);
\draw[ultra thick,cyan] (2.5,5) -- (5+\a,5);
\draw[ultra thick,cyan] (2.5,0) -- (-\a,0);
\end{tikzpicture}
\caption{A 2-dimensional cubical complex.}
\label{fig:kb}
\end{figure}
\label{eg:kb}
\end{eg}

\subsection{Discrete singular cubical homology}

Let $Q_n$ denote the \emph{$n$-cube graph}, defined by
\[Q_n=\begin{cases} I_0&\mbox{if }n=0\\ I_1^n&\mbox{if }n\geq 1.\end{cases}\]
For $n\geq 0$, let $\mathcal{L}_n(G)$ denote the free abelian group generated by all graph maps $f:Q_n\to G$. For $1\leq k\leq n$, let $D_k^\pm f:Q_{n-1}\to G$ be the graph maps given by
\begin{align*}
D_k^-f(i_1,\ldots,i_{n-1}) &= f(i_1,\ldots,i_{k-1},0,i_k,\ldots,i_n)\\
D_k^+f(i_1,\ldots,i_{n-1}) &= f(i_1,\ldots,i_{k-1},1,i_k,\ldots,i_n).
\end{align*}
A map $f$ is called \emph{degenerate} if $D_k^-f=D_k^+f$ for some $k$. By definition, no map $Q_0\to g$ is degenerate. Let $\mathcal{D}_n(G)$ denote the subgroup of $\mathcal{L}_n(G)$ generated by all degenerate maps, and let $\mathcal{C}_n(G)=\mathcal{L}_n(G)/\mathcal{D}_n(G)$. Define a graph map $\partial_n f$ by
\[\partial_n f = \sum_{k=1}^n (-1)^n (D_k^- f - D_k^+ f).\]
Extending linearly, we obtain a homomorphism $\partial_n : \mathcal{C}_n(G)\to \mathcal{C}_{n-1}(G)$ for each $n\geq 1$. It is routine to check that $(\mathcal{C}_\bullet,\partial_\bullet)$ is a chain complex.

\begin{mydef}
The \emph{$n$th discrete singular cubical homology group} of $G$ is the quotient $\mathcal{H}_n(G)=\operatorname{ker}\partial_n/\operatorname{im}\partial_{n+1}$.
\end{mydef}

\begin{prop}[{\cite[Theorem 4.1]{barcelo2014}}]
The first discrete singular cubical homology group $\hc_1(G)$ is isomorphic to the abelianization of the discrete fundamental group $A_1(G)$.
\label{prop:hur}
\end{prop}

\begin{eg}
Let $G$ be the graph from Example \ref{eg:kb}. Proposition \ref{prop:hur} gives $\hc_1(G)\cong \Z\oplus \Z_2$.
\end{eg}
\section{Reduced products and token graphs}
\label{sec:tok}

As in the introduction, let $\rp{G}{n}$ denote the quotient graph of $G^n$ under the action of $\Sigma_n$. The vertices of the \emph{$n$th reduced power} $\rp{G}{n}$ correspond to configurations of $n$ indistinguishable tokens placed on the vertices of $G$, with multiple tokens allowed on each vertex. Two such configurations are adjacent if and only if they differ by moving exactly one token to an adjacent vertex of $G$.

We represent each configuration of tokens by a monomial in the vertices of $G$, where the multiplicity of a vertex is the number of tokens on that vertex. For example, if the vertices of $G=K_3$ are labeled $u$, $v$ and $w$, then the monomial $u^2v$ corresponds to the configuration with 2 tokens on $u$, 1 token on $v$ and no tokens on $w$. Two monomials $\x$ and $\y$ of degree $n$ are adjacent as vertices of $\rp{G}{n}$ if and only if
\[\frac{\operatorname{lcm}(\x,\y)}{\operatorname{gcd}(\x,\y)} = uv\]
for adjacent vertices $u$ and $v$ of $G$.

The \emph{token graph} $\tok{G}{n}$ is the subgraph of $\rp{G}{n}$ induced by all configurations with at most one token at each vertex. Equivalently, $\tok{G}{n}$ is induced by the squarefree monomials. Clearly $\rp{G}{1}=\tok{G}{1}=G$, and $\tok{G}{n}$ is empty if $n>t$, where $t$ is the number of vertices of $G$. Moreover we have $\rp{G}{n}\cong \rp{G}{t-n+1}$ for all $n$. Since $G$ is connected, both $\rp{G}{n}$ and $\tok{G}{n}$ are connected as well \cite[Theorem 5]{fabmon2012}.

We regard a path in $G$ as a sequence $P=(p_0,\ldots,p_\ell)$ of vertices with $p_i\simeq p_{i+1}$ for all $i$. Thus $P$ is a \emph{cycle} if $p_0=p_\ell$. The \emph{length} of $P$ is defined to be $\ell$.

\begin{mydef}
A vertex $v\in G$ is \emph{essential} if $\deg v\neq 2$. We say that $G$ is \emph{sufficiently subdivided for $n$} if it satisfies the following conditions:
\begin{enumerate}
\item Every path between distinct essential vertices of $G$ has length at least $n-1$
\item Every cycle based at an essential vertex of $G$ that is not nullhomotopic (when regarding $G$ as a topological space) has length at least $n+1$.
\end{enumerate}
\end{mydef}

Since $G$ is simple, it is always sufficiently subdivided for 2. Clearly if $G$ is sufficiently subdivided for $n$, then it is sufficiently subdivided for all $m\leq n$. For any fixed $G$ and $n$, one can subdivide the edges of $G$ to obtain a graph that is homeomorphic to $G$ and sufficiently subdivided for $n$. This process, illustrated in Figure \ref{fig:subd}, does not affect the braid groups.

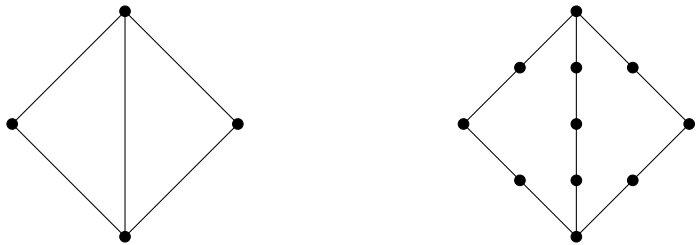
\begin{figure}[ht]
\centering
\begin{tikzpicture}
\def\a{1.5}
\def\b{-6}
\draw (\b,\a)--(\b+\a,0)--(\b,-\a)--(\b-\a,0)--cycle;
\draw (\b,-\a)--(\b,\a);
\draw[fill=black] (\b,-\a) circle (2pt);
\draw[fill=black] (\b,\a) circle (2pt);
\draw[fill=black] (\b+\a,0) circle (2pt);
\draw[fill=black] (\b-\a,0) circle (2pt);
\draw (0,\a)--(\a,0)--(0,-\a)--(-\a,0)--cycle;
\draw (0,-\a)--(0,\a);
\draw[fill=black] (0,-\a) circle (2pt);
\draw[fill=black] (0,\a) circle (2pt);
\draw[fill=black] (\a,0) circle (2pt);
\draw[fill=black] (-\a,0) circle (2pt);
\draw[fill=black] (0,\a/2) circle (2pt);
\draw[fill=black] (0,-\a/2) circle (2pt);
\draw[fill=black] (0,0) circle (2pt);
\draw[fill=black] (\a/2,\a/2) circle (2pt);
\draw[fill=black] (-\a/2,\a/2) circle (2pt);
\draw[fill=black] (-\a/2,-\a/2) circle (2pt);
\draw[fill=black] (\a/2,-\a/2) circle (2pt);
\end{tikzpicture}
\caption{A graph sufficiently subdivided for at most 2, left, and a homeomorphic graph sufficiently subdivided for at most 5, right.}
\label{fig:subd}
\end{figure}

\begin{eg}[Paths]
Consider the path graph $I_m$. Let $\Delta_{m,n}$ denote the subgraph of $\Z^m$ induced by the set
\[\{x\in \Z^m : 0\leq x_1\leq \cdots \leq x_m \leq n\}.\]
Thus $\Delta_{m,n}$ consists of the integer points in an $m$-simplex. We identify $\rp{I_m}{n}$ with $\Delta_{m,n}$ as follows. As described above, we regard the vertices of $\rp{I_m}{n}$ as monomials $\x$. For each $i\in I_m$, let $\x(i)$ denote the multiplicity of $i$ in $\x$. For $k=1,\ldots,m$, let $\phi_k : \rp{I_m}{n}\to \Z$ be given by
\[\phi_k(\x) = \sum_{i=m-k+1}^m \x(i).\]
Let $\phi: \rp{I_m}{n}\to \Delta_{m,n}$ be given by $\phi(\x)=(\phi_1(\x),\ldots,\phi_m(\x))$. It can be shown that $\phi$ is a graph isomorphism.

We can describe $\tok{I_m}{n}$ similarly. Let $\Gamma_{n,m}$ be the subgraph of $\Z^n$ induced by the set
\[\{x\in \Z^m : 0\leq x_1<\cdots < x_n\leq m\}.\]
The $n$-token configurations on $I_m$ with at most one token on each vertex can be identified with the strictly increasing functions $\{1,\ldots,n\}\to \{0,\ldots,m\}$, or equivalently with the vertices of $\Gamma_{n,m}$. This identification gives an isomorphism $\tok{I_m}{n}\cong \Gamma_{n,m}$. Additionally we have $\Delta_{m,n} + (0,\ldots,m-1)= \Gamma_{m,n+m-1}$, so in fact
\[\rp{I_m}{n}\cong \tok{I_{n+m-1}}{m}.\]
Using these descriptions and Proposition \ref{prop:dfg}, one can show that $A_1(I_m)$, $A_1(\rp{I_m}{n})$ and $A_1(\tok{I_m}{n})$ are trivial for all $n$. Since $G$ is sufficiently subdivided for $m$, Theorem \ref{thm:tokbraid} implies that the braid group $B_n(I_m)$ is trivial whenever $n\leq m$. It is not hard to see that $\uc{I_m}{n}$ is an $n$-simplex, so in fact $B_n(I_m)$ is trivial for all $n$.

\label{eg:spim}
\end{eg}

\begin{figure}[ht]
\begin{tikzpicture}
\def\a{2.5}
\draw (0,0) -- (0,1) -- (1,1);
\draw (\a+1,0) -- (\a+1,2) -- (\a+3,2);
\draw (\a+1,1) -- (\a+2,1) -- (\a+2,2);
\draw (2*\a+3,0) -- (2*\a+3,3) -- (2*\a+6,3);
\draw (2*\a+3,1) -- (2*\a+4,1) -- (2*\a+4,3);
\draw (2*\a+3,2) -- (2*\a+5,2) -- (2*\a+5,3);
\draw[fill=black] (0,0) circle (2pt);
\draw[fill=black] (0,1) circle (2pt);
\draw[fill=black] (1,1) circle (2pt);
\draw[fill=black] (\a+1,0) circle (2pt);
\draw[fill=black] (\a+1,2) circle (2pt);
\draw[fill=black] (\a+3,2) circle (2pt);
\draw[fill=black] (\a+1,1) circle (2pt);
\draw[fill=black] (\a+2,1) circle (2pt);
\draw[fill=black] (\a+2,2) circle (2pt);
\draw[fill=black] (2*\a+3,0) circle (2pt);
\draw[fill=black] (2*\a+3,3) circle (2pt);
\draw[fill=black] (2*\a+6,3) circle (2pt);
\draw[fill=black] (2*\a+3,1) circle (2pt);
\draw[fill=black] (2*\a+4,1) circle (2pt);
\draw[fill=black] (2*\a+4,3) circle (2pt);
\draw[fill=black] (2*\a+3,2) circle (2pt);
\draw[fill=black] (2*\a+5,2) circle (2pt);
\draw[fill=black] (2*\a+5,3) circle (2pt);
\draw[fill=black] (2*\a+4,2) circle (2pt);
\end{tikzpicture}
\caption{From left to right: the graphs $\rp{I_2}{1}\cong \tok{I_2}{2}$, $\rp{I_2}{2}\cong \tok{I_3}{2}$ and $\rp{I_2}{3}\cong \tok{I_4}{2}$.}
\end{figure}
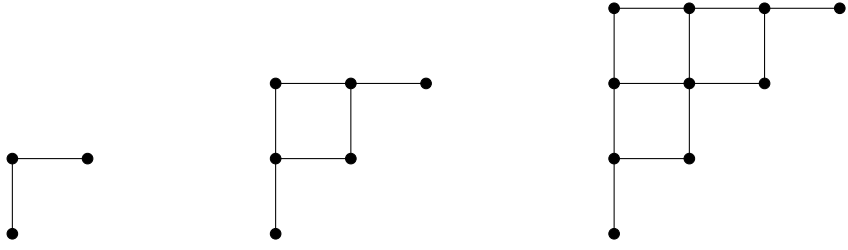

\begin{eg}
Let $G$ be the graph on the left of Figure \ref{fig:3gra}. Proposition \ref{prop:dfg} implies that $A_1(G)\cong F_2$, where $F_n$ is the free group of rank $n$. Theorem \ref{thm:symab} then gives $A_1(\rp{G}{n})\cong \Z^2$ for all $n\geq 2$. With the help of \texttt{SageMath} \cite{sagemath}, we can compute
\[A_1(\tok{G}{n})\cong \begin{cases} \Z*(F_{n-1}\oplus \Z)&\mbox{if }1\leq n\leq 4\\ F_4\oplus \Z&\mbox{if }n=5\\ \Z*(F_{9-n}\oplus \Z)&\mbox{if }6\leq n\leq 9,\end{cases}\]
where $*$ denotes the free product. For $n\geq 3$, we should not expect these to resemble the braid groups $B_n(G)$, since $G$ is not sufficiently subdivided for these values of $n$.
\end{eg}

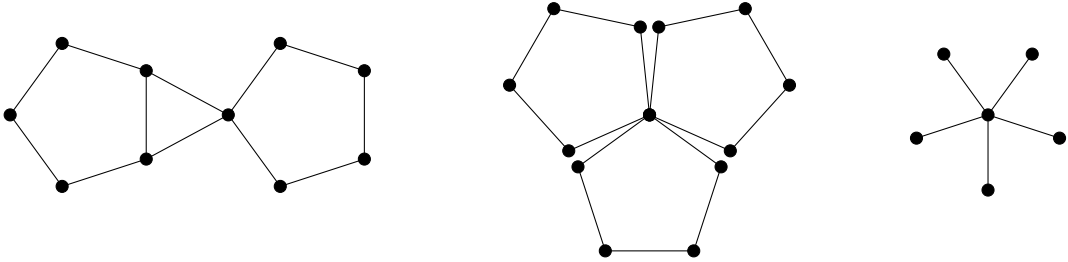
\begin{figure}[ht]
\begin{tikzpicture}
\def\a{2.9}
\def\b{7.5}
\def\c{4.5}
\def\r{0.08}
\foreach\x in {0,...,4}{
\draw[fill=black] (36+72*\x:1) circle (\r);
\draw (36+72*\x:1) -- (108+72*\x:1);
\draw[fill=black] ($(\a,0)+(36+72*\x:1)$) circle (\r);
\draw ($(\a,0)+(36+72*\x:1)$) -- ($(\a,0)+(108+72*\x:1)$);
}
\draw (36:1) -- ($(\a,0)+(180:1)$) -- (324:1);

\coordinate (a) at ($(\b,0)+(30:1)$);
\coordinate (b) at ($(\b,0)+(150:1)$);
\coordinate (c) at ($(\b,0)+(270:1)$);

\foreach\x in {0,...,4}{
\draw[fill=black] ($(a)+(210+72*\x:1)$) circle (\r);
\draw ($(a)+(210+72*\x:1)$) -- ($(a)+(282+72*\x:1)$);
\draw[fill=black] ($(b)+(330+72*\x:1)$) circle (\r);
\draw ($(b)+(330+72*\x:1)$) -- ($(b)+(42+72*\x:1)$);
\draw[fill=black] ($(c)+(90+72*\x:1)$) circle (\r);
\draw ($(c)+(90+72*\x:1)$) -- ($(c)+(162+72*\x:1)$);
}

\coordinate (d) at (\b+\c,0);
\draw[fill=black] (d) circle (\r);
\foreach\x in {0,...,4}{
\draw (d) -- ($(d)+(270+72*\x:1)$);
\draw[fill=black] ($(d)+(270+72*\x:1)$) circle (\r);
}
\end{tikzpicture}
\caption{Three graphs.}
\label{fig:3gra}
\end{figure}

\begin{eg}[Bouquets of cycles]
Let $G$ be the graph obtained by identifying a single vertex on each of $k$ disjoint $m$-cycle graphs, where $m\geq 5$. In other words, $G$ is a ``wedge sum'' of $m$-cycles. The case $k=3$ and $m=5$ is illustrated in the center of Figure \ref{fig:3gra}. Proposition \ref{prop:dfg} implies that $A_1(G)\cong F_k$, so Theorem \ref{thm:symab} gives $A_1(\rp{G}{n})\cong \Z^k$ for all $n\geq 2$. Theorem \ref{thm:tokbraid} implies that $A_1(\tok{G}{n})\cong B_n(G)$ for all $n<m$. It is shown in \cite[Proposition 3.4]{kallel2016} that $B_2(G)$ is a free group of rank
\begin{equation}
3\binom{k}{2}+1.
\label{eq:2rk}
\end{equation}
For all $n$, the general result \cite[Theorem 3.16]{ko2012} implies that the first singular homology group $H_1(\uc{G}{n})$ is a free abelian group of rank
\begin{equation}
(2n-1)\binom{n+k-2}{n}+1.
\label{eq:nrk}
\end{equation}
This agrees with \eqref{eq:2rk} in the case $n=2$, since $G$ is sufficiently subdivided for 2. Computations suggest that $B_n(G)$ is a free group of rank \eqref{eq:nrk} whenever $n<m$, but we have not found a proof of this in the literature.
\end{eg}

\begin{eg}[Stars]
Let $S_m$ denote the star graph on $m+1$ vertices. For example, $S_5$ is pictured on the right of Figure \ref{fig:3gra}. Let $\Theta_{m,n}$ be the subgraph of $\Z^m$ induced by the vertex set
\[\{x\in \Z^m : x_1+\cdots + x_m\leq n\mbox{ and }x_i\geq 0\mbox{ for all } i\}.\]
This is the set of integer points of an $m$-simplex. Label the internal node of $S_m$ $0$ and the leaves $1,\ldots,m$. Regarding the vertices of $\rp{S_m}{n}$ as monomials $\x$, let $\x(i)$ denote the multiplicity of each $i\in S_m$ in $\x$. Let $\psi: \rp{S_m}{n}\to \Theta_{m,n}$ be given by $\psi(\x)=(\x(1),\ldots,\x(m))$. It is routine to check that $\psi$ is a graph isomorphism. Under this isomorphism, the token graph $\tok{S_m}{n}$ is the subgraph of the $m$-cube $Q_m$ induced by all vertices whose coordinates sum to $n-1$ or $n$. It follows that $\tok{S_m}{n}$ is a $(m,n)$-biregular graph on $\binom{m+1}{n}$ vertices.

Proposition \ref{prop:dfg} and Theorem \ref{thm:symab} imply that $A_1(S_m)$ and $A_1(\rp{S_m}{n})$ are trivial. However, $A_1(\tok{S_m}{n})$ is often nontrivial. Using ideas from Section \ref{sec:tokbraid}, it is not hard to show that $A_1(\tok{S_m}{n})$ is free. We propose the following formula for its rank, which we have checked for $m\leq 11$ and $n\leq 7$ using \texttt{SageMath} \cite{sagemath}:

\begin{conj}
The free group $A_1(\tok{S_m}{n})$ is of rank
\begin{equation}
(n-1)\binom{m}{n} - \binom{m}{n-1} + 1.
\label{eq:starconj}
\end{equation}
\label{conj:star}
\end{conj}

It is proven in \cite[Corollary 4.2]{farley2005} that $B_n(S_m)$ is a free group of rank\footnote{A much larger value for $\operatorname{rk} B_n(S_m)$ appears in \cite[Proposition 4.1]{ghrist2001}. Computations support \eqref{eq:starko}.}
\begin{equation}
(m-2)\binom{n+m-2}{n-1} - \binom{n+m-2}{n} + 1.
\label{eq:starko}
\end{equation}
When $n=2$, this formula agrees with \eqref{eq:starconj}, since $S_m$ is sufficiently subdivided for 2.
However, $S_m$ is not sufficiently subdivided for $n\geq 3$, and the formulas differ in these cases.
\end{eg}
\section{Proof of Theorem \ref{thm:symab}}
\label{sec:symab}

Order the vertices of $G$ as $(v_0,v_1,\ldots,v_m)$. Let $C_1(G)$ denote the free abelian group with basis $\{v_iv_j : i<j\mbox{ and }v_i\sim v_j\}$. In other words, $C_1(G)$ is the group of 1-chains of $G$ if we regard $G$ as a 1-complex. Given vertices $v_i\simeq v_j$ of $G$, define a 1-chain $[v_i,v_j]$ by
\[[v_i,v_j] = \begin{cases} v_iv_j&\mbox{if }i<j\\ -v_jv_i&\mbox{if }i>j\\ 0&\mbox{if }i=j.\end{cases}\]
Given a path $P=(p_0,\ldots,p_\ell)$ in $G$, we write
\[[P] = \sum_{i=0}^{\ell-1} [p_i,p_{i+1}].\]
In this notation, $P$ is a cycle when $p_0=p_\ell$. Let $H_1(G)$ denote the subgroup of $C_1(G)$ generated by $\{[C] : C\mbox{ is a cycle of }G\}$. This is the first simplicial homology group of $G$.

Let $\x\in \rp{G}{n-1}$, where we think of $\rp{G}{0}$ as consisting only of the identity monomial $1$. Let $G\x$ denote the subgraph of $\rp{G}{n}$ induced by all monomials divisible by $\x$. It is not hard to see that $G\x$ is isomorphic to $G$. The following proposition will play an important role in proving Theorem \ref{thm:symab}.

\begin{prop}
Let $n\geq 2$. There is a set $S$ of 4-cycles of $\rp{G}{n}$ such that for any $\x\in \rp{G}{n-1}$ we have
\[H_1(\rp{G}{n}) = H_1(G\x)\oplus \mathcal{S},\]
where $\mathcal{S}$ is generated by $\{[C] : C\in S\}$.
\label{prop:hombasis}
\end{prop}

Proposition \ref{prop:hombasis} is an integral version of \cite[Theorem 1]{hammack2016}, which describes a \emph{cycle basis} of $\rp{G}{n}$, i.e. a basis of $H_1(\rp{G}{n};\Z_2)$. It turns out that this is also a basis of $H_1(\rp{G}{n})=H_1(\rp{G}{n};\Z)$. Most of the relevant arguments from \cite{hammack2016} apply in the integral case without any changes. The lone exception is \cite[Lemma 1]{hammack2016}, which we adapt as Lemma \ref{lem:hammack} below.

We can construct generating sets of $H_1(G\x)$ and $\mathcal{S}$ as follows. Let $T$ be a spanning tree of $G$. For each edge $e$ of $G$ not in $T$, choose a cycle $C_e$ such that $C_e$ contains $e$, and every edge of $C_e\setminus e$ is contained in $T$. It is well known that the 1-chains $[C_e]$ form a basis for $H_1(G)$. Any cycle $C=(p_0,\ldots,p_\ell)$ of $G$ corresponds to a cycle $C\x$ in $G\x$ given by $C\x = (p_0\x,\ldots,p_\ell\x)$. Clearly the 1-chains $[C_e\x]$ form a basis for $H_1(G\x)$.

We can take the set $S$ in Proposition \ref{prop:hombasis} to be the set of \emph{Cartesian squares} of $\rp{G}{n}$, which we now define.

\begin{mydef}
Let $n\geq 2$. Given distinct edges $ab$ and $cd$ of $G$ and a monomial $\x\in \rp{G}{n-2}$, there is a simple 4-cycle $(ab\,\square\,cd)\x$ of $\rp{G}{n}$ given by
\[(ab\,\square\,cd)\x=(ac\x,ad\x,bd\x,bc\x,ac\x).\]
A 4-cycle of this form is called a \emph{Cartesian square} of $\rp{G}{n}$.
\label{mydef:cartsq}
\end{mydef}

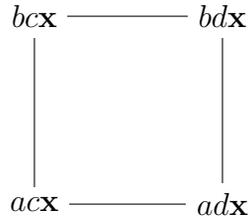
\begin{figure}[ht]
\begin{tikzpicture}[scale=2.5]
\draw (0,0) node[fill=white] {$ac\x$} -- (1,0) node[fill=white] {$ad\x$} -- (1,1) node[fill=white] {$bd\x$} -- (0,1) node[fill=white] {$bc\x$} --cycle;
\end{tikzpicture}
\caption{A Cartesian square of $\rp{G}{n}$.}
\end{figure}

Any graph map $f:G\to H$ induces a homomorphism $f^*:C_1(G)\to C_1(H)$ by setting $f^*([(u,v)])=[(f(u),f(v))]$ and extending linearly. If $C=(v_0,\ldots,v_k)$ is a cycle of $G$, then clearly $f(C)=(f(v_0),\ldots,f(v_k))$ is a cycle of $H$. It follows that $f^*$ restricts to a homomorphism $ H_1(G)\to H_1(H)$ given by $f^*([C])=[f(C)]$.

\begin{lemma}
Let $\eta:G^n \to \rp{G}{n}$ be the graph map given by
\[\eta(v_1,\ldots,v_n) = \prod_{i=1}^n v_i.\]
The induced homomorphism $\eta^*:  H_1(G^n)\to  H_1(\rp{G}{n})$ is surjective.

\begin{proof}
We first claim that if two vertices of $G^n$ differ by a permutation of their coordinates, then there is a path in $G^n$ between these two vertices whose image under $\eta^*\circ [\cdot]$ is 0. Let $a$ and $b$ be vertices of $G$. Consider an arbitrary vertex of $G^n$ that contains $a$ and $b$ among its coordinates; write this arbitrary vertex as $(\ldots,a,\ldots,b,\ldots)$. Let $P=(p_0,\ldots,p_\ell)$ be a path in $G$ from $a$ to $b$. Let $Q$ be the path in $G^n$ given by
\[Q=((\ldots,p_0,\ldots,b,\ldots),(\ldots,p_1,\ldots,b,\ldots),\ldots,(\ldots,p_\ell,\ldots,b,\ldots)).\]
Let $R$ be the path in $G^n$ given by
\[R=((\ldots,a,\ldots,p_\ell,\ldots),(\ldots,a,\ldots,p_{\ell-1},\ldots),\ldots,(\ldots,a,\ldots,p_0,\ldots)).\]
The concatenation $Q+R$ is a path in $G^n$ from $(\ldots,a,\ldots,b,\ldots)$ to $(\ldots,b,\ldots,a,\ldots)$. Moreover, we have 
\begin{align*}
[Q+R] &= [Q] + [(\ldots,b,\ldots,b,\ldots),(\ldots,b,\ldots,b,\ldots)] + [R]\\
&= [Q]+[R].
\end{align*}
Since $\eta(Q)$ is the reverse of $\eta(R)$, we have $\eta^*([Q]) = -\eta^*([R])$, so $\eta^*([Q+R])=0$. Thus the claim holds for transpositions of coordinates. Since any permutation is a product of transpositions, the claim follows.

Let $C=(\x_0,\ldots,\x_k)$ be a cycle of $\rp{G}{n}$. We construct a cycle $Z$ of $G^n$ such that $\eta^*([Z])=[C]$. For each $i=0,\ldots,k-1$ let $(\u_i,\v_{i+1})$ be a pair of adjacent vertices of $G^n$ such that $f^*([(\u_i,\v_{i+1})]) = [(\x_i,\x_{i+1})]$. For $i=0,\ldots,k-1$, the vertices $\u_i$ and $\v_i$ differ by a permutation of their coordinates, so by the claim above, there is a path $P_i$ from $\v_i$ to $\u_i$ such that $\eta^*([P_i])=0$. Let $Z$ be the following cycle of $G^n$:
\[Z = (\u_0,\v_1) + P_1 + (\u_1,\v_2) + P_2 + \cdots + (\u_{k-2},\v_{k-1}) + P_{k-1} + (\u_{k-1},\v_k).\]
We have $\eta^*([Z])=[C]$, as desired.
\end{proof}
\label{lem:hammack}
\end{lemma}

\begin{proof}[Proof of Proposition \ref{prop:hombasis}]
We construct a homomorphism $\phi:H_1(\rp{G}{n})\to H_1(G)$, following \cite{hammack2016}. Let $\x$ and $\y$ be adjacent vertices of $\rp{G}{n}$, so that $\x/\y=u/v$ for some $u,v\in G$. Let $\phi([\x,\y]) = [u,v]$ for any such $\x$ and $\y$, and extend linearly to obtain a homomorphism $C_1(\rp{G}{n})\to C_1(G)$. Moreover, if $p_i:G^n\to G$ denotes the projection onto the $i$th coordinate, then $\phi\circ \eta^* = \sum_{i=1}^n p_i^*$.

Let $X\in H_1(\rp{G}{n})$. Lemma \ref{lem:hammack} gives $Y\in H_1(\rp{G}{n})$ such that $\eta^*(Y)=X$, so we have $\phi(X)=\phi(\eta^*(Y)) = \sum_{i=1}^n p_i^*(Y)\in H_1(G)$. Thus $\phi$ is a function $H_1(\rp{G}{n})\to H_1(G)$. In particular, for any $\x\in \rp{G}{n-1}$ the restriction $H_1(G\x)\to H_1(G)$ is an isomorphism, so
\[H_1(\rp{G}{n}) = H_1(G\x)\oplus \operatorname{ker} \phi.\]
It follows from \cite[Proposition 5]{hammack2016} that $\operatorname{ker} \phi$ is generated by the set of $[C]$ for all Cartesian squares $C$ of $\rp{G}{n}$. The arguments used do not rely on $\Z_2$ and hold over $\Z$ as well. The result follows.
\end{proof}

\begin{lemma}
For any graphs $G$ and $H$, we have $A_1(G\times H)\cong A_1(G)\times A_1(H)$.
\begin{proof}
Let $(v_0,w_0)\in G\times H$. A function $f:\Z\to G\times H$ is a graph map $f:(\Z,\partial \Z)\to (G\times H,(v_0,w_0))$ if and only if, writing $f(i)=(f_G(i),f_H(i))$, the functions $f_G$ and $f_H$ are graph maps $f_G:(\Z,\partial \Z)\to (G,v_0)$ and $f_H:(\Z,\partial \Z)\to (H,w_0)$. Similarly, a based homotopy $h$ from $f$ to $g$ induces based homotopies $h_G$ from $f_G$ to $g_G$ and $h_H$ from $f_H$ to $g_H$.

Let $f':(\Z,\partial\Z)\to (G\times H,(v_0,w_0))$ be another graph map, and let $\phi_1:\Z\times I_{m_1}$ and $\phi_2:\Z\times I_{m_2}$ be based homotopies from $f_G$ to $f_G'$ and from $f_H$ to $f_H'$, respectively. We show that these induce a based homotopy from $f$ to $f'$. We cannot simply take the function $(\phi_1,\phi_2)$, since this is not necessarily a graph map. Instead, define $\phi:\Z\times I_{m_1+m_2}\to G\times H$ as follows:
\[\phi(i,j) = \begin{cases} (\phi_1(i,j),f_H(i))&\mbox{if }0\leq j\leq m_1\\ (f_G'(i),\phi_2(i,j-m_1))&\mbox{if }m_1\leq j\leq m_1+m_2.\end{cases}\]
It is routine to check that $\phi$ is a based homotopy from $f$ to $f'$. We therefore obtain a bijection from $A_1(G\times H)$ to $A_1(G)\times A_1(H)$. This bijection is clearly a homomorphism, and hence an isomorphism.
\end{proof}
\label{lem:prod}
\end{lemma}

In what follows, we fix a base vertex $v_0\in G$ and write $\v_0=(v_0,\ldots,v_0)\in G^n$ and $\x_0=v_0^n\in \rp{G}{n}$. We will consider these as the base vertices of $G^n$ and $\rp{G}{n}$, respectively, so that $A_1(G)=A_1(G,v_0)$, $A_1(G^n)=A_1(G^n,\v_0)$ and $A_1(\rp{G}{n}) = A_1(\rp{G}{n},\x_0)$. Note that $\eta(\v_0) = \x_0$, so $\eta$ induces a homomorphism $\eta_*:A_1(G^n) \to A_1(\rp{G}{n})$.

\begin{lemma}
The homomorphism $\eta_*:A_1(G^n) \to A_1(\rp{G}{n})$ is surjective.

\begin{proof}
Let $\x_0$ and $\v_0$ be defined as above. Let $C=(\x_0,\ldots,\x_k)$ be a cycle of $\rp{G}{n}$. It will suffice to construct a cycle $Z=(\v_0,\ldots,\v_k)$ of $G^n$ such that $\eta(Z)=C$. Assume without loss of generality that $\x_i$ and $\x_{i+1}$ are distinct for all $i$. The quotient $\x_{i+1}/\x_i$ has the form $x_v/x_u$ for some adjacent vertices $u,v\in G$. For $i=0,\ldots,k-1$ in order, at least one coordinate of $\v_i$ is $u$; choose one such coordinate and replace it with $v$ to obtain $\v_{i+1}$. By construction we have $\v_k=\v_0$ and $\eta(\v_i)=\x_i$ for all $i$. Thus $Z$ is the desired cycle of $G^n$.
\end{proof}
\label{lem:etasurj}
\end{lemma}

The argument for the following proposition is essentially due to P.~A.~Smith \cite{smith1936}.

\begin{prop}
For all $n\geq 2$ the group $A_1(\rp{G}{n})$ is abelian.

\begin{proof}
Let $n\geq 2$ and $\alpha,\beta\in A_1(G)$. Let $\alpha^{(i)}\in A_1(G)^n$ denote the element whose $i$th coordinate is $\alpha$ and whose other coordinates are all 1. Define $\beta^{(i)}$ similarly. We consider these as elements of $A_1(G^n)$ under the isomorphism $A_1(G^n)\cong A_1(G)^n$ afforded by Lemma \ref{lem:prod}. Thus $A_1(G^n)$ is generated by all elements of the form $\alpha^{(i)}$, as alpha runs over $A_1(G)$ and $i$ runs over all indices. Since $\eta_*:A_1(G^n) \to A_1(\rp{G}{n})$ is surjective by Lemma \ref{lem:etasurj}, it suffices to show that $\eta_*(\alpha^{(i)})$ and $\eta_*(\beta^{(j)})$ commute in $A_1(\rp{G}{n})$ for any $i$ and $j$. Note that $\eta_*(\alpha^{(i)})=\eta_*(\alpha^{(j)})=\alpha$, and similarly with $\beta$. Hence
\[\eta_*(\alpha^{(i)})\eta_*(\beta^{(j)}) = \eta_*(\alpha^{(1)})\eta_*(\beta^{(2)}) = \eta_*(\alpha^{(1)}\beta^{(2)}).\]
Note that in $A_1(G^n)$ we have $\alpha^{(1)}\beta^{(2)} = (\alpha,\beta,\cdots)  =\beta^{(2)}\alpha^{(1)}$, so
\[\eta_*(\alpha^{(1)}\beta^{(2)}) = \eta_*(\beta^{(2)}\alpha^{(1)}) = \eta_*(\beta^{(2)})\eta_*(\alpha^{(1)}) = \eta_*(\beta^{(j)})\eta_*(\alpha^{(i)}),\]
and we are done.
\end{proof}
\label{prop:smith}
\end{prop}

\begin{proof}[Proof of Theorem \ref{thm:symab}]
Recall the CW complex $X(G)$ from Proposition \ref{prop:dfg}. Propositions \ref{prop:hur} and \ref{prop:smith} together imply that $A_1(\rp{G}{n})\cong H_1(X(\rp{G}{n}))$. To construct $X(\rp{G}{n})$, we attach a 2-cell to every Cartesian square of $\rp{G}{n}$. Hence Proposition \ref{prop:hombasis} gives $H_1(X(\rp{G}{n})) = H_1(X(G\x))$. We have $H_1(X(G\x))\cong \hc_1(G\x) \cong \hc_1(G)$. Putting everything together, we have $A_1(\rp{G}{n})\cong \hc_1(G)$, as desired.
\end{proof}
\section{Discrete homotopy of token graphs}
\label{sec:tokbraid}

We first prove Theorem \ref{thm:tokbraid}. Then, to shed light on the separate case in which $G$ contains 3- or 4-cycles, we discuss the meaning of \emph{local exchanges} mentioned in the introduction. Throughout this section, we assume that $G$ is sufficiently subdivided for $n$.

\subsection{Proof of Theorem \ref{thm:tokbraid}}

We will briefly regard $G$ as a 1-complex, so that $G$ consists of 0-cells (vertices) and 1-cells (edges). If $c$ is a 0-cell, we write $\partial c=\{c\}$; if $c$ is a 1-cell, we write $\partial c$ for the set of endpoints of $c$. Abrams \cite{abrams2000} defines the \emph{discrete configuration space} $\ud{G}{n}$ as follows:
\[\ud{G}{n} = \{\{c_1,\ldots,c_n\}\subset G : \partial c_i\cap \partial c_j = \emptyset\mbox{ if }i\neq j\}.\]
Here we consider the elements of $G$ to be cells, so that each $c_i$ is a cell. Thus $\ud{G}{n}$ is a cubical complex, i.e. a polyhedral complex whose cells are cubes of various dimensions \cite[Definition 2.42]{kozlov2008}. The dimension of the cell $\{c_1,\ldots,c_n\}$ of $\ud{G}{n}$ is the number of 1-cells among the $c_i$. The vertices of the cell $\{c_1,\ldots,c_n\}$ are the sets $\{d_1,\ldots,d_n\}$ with $d_i\in \partial c_i$ for all $i$. Similar statements apply to $\ud{G}{n}$, replacing tuples with sets.

\begin{prop}[{\cite{abrams2000, kim2012, prue2014}}]
If $G$ is sufficiently subdivided for $n$, then $\ud{G}{n}$ is a deformation retract of $\uc{G}{n}$.
\label{prop:abrams}
\end{prop}

Since $\uc{G}{n}$ is not compact in general, it does not have the structure of a finite CW complex. Proposition \ref{prop:abrams} implies that $\uc{G}{n}$ is in fact homotopy equivalent to the much nicer space $\ud{G}{n}$. Given a CW complex $K$, let $\sk_n K$ denote its $n$-skeleton. Recall the 2-complex $X(G)$ from Proposition \ref{prop:dfg}, obtained by attaching a 2-cell to each 3-cycle and chordless 4-cycle of $G$. Our proof of Theorem \ref{thm:tokbraid} relies on the following observation.

\begin{prop}
We have $\tok{G}{n} \cong \sk_1 \ud{G}{n}$ and, if $G$ contains no 3- or 4-cycles, then $K(\tok{G}{n})\cong\sk_2 \ud{G}{n}$.

\begin{proof}
The 0-cells of $\ud{G}{n}$ are sets of $n$ distinct vertices of $G$. These sets are clearly in bijection with the vertices of $\tok{G}{n}$. Two 0-cells are the endpoints of a 1-cell in $\ud{G}{n}$ if and only if their symmetric difference has the form $\{u,v\}$ for adjacent vertices $u$ and $v$ of $G$. This occurs precisely when, writing $\x$ and $\y$ for the corresponding monomials, we have either $\x/\y = u/v$ or $\x/\y=v/u$. But this is true for some such $u$ and $v$ if and only if $\x$ and $\y$ are adjacent in $\tok{G}{n}$. Hence $\tok{G}{n}\cong \sk_1 \ud{G}{n}$.

Suppose that $G$ contains no 3- or 4-cycles. It follows that the reduced power $\rp{G}{n}$ contains no 3-cycles. Moreover the only 4-cycles of $\rp{G}{n}$ are the Cartesian squares from Definition \ref{mydef:cartsq}. Consider an arbitrary Cartesian square $(e_1\,\square\,e_2)\x$ of $\rp{G}{n}$, and write $\x=v_1\cdots v_{n-1}$. This Cartesian square is contained in the subgraph $\tok{G}{n}$ if and only if the following conditions hold:
\begin{enumerate}
	\item The edges $e_1$ and $e_2$ have no endpoints in common
	\item $\x$ is not divisible by any endpoints of $e_1$ or $e_2$
	\item $\x$ is squarefree.
\end{enumerate}
If these conditions hold, then the Cartesian square gives rise to a 2-cell $\{e_1,e_2,v_1,\ldots,v_{n-1}\}$ of $\ud{G}{n}$. Conversely, let $\{f_1,f_2,w_1,\ldots,w_{n-1}\}$ be a 2-cell of $\ud{G}{n}$, where the $f_i$ are 1-cells and the $w_i$ are 0-cells. The condition $\partial f_1\cap \partial f_2 = \emptyset$ means that $f_1$ and $f_2$ have no endpoints in common. The condition $\partial f_i\cap \partial w_j=\emptyset$ means that $\y=w_1\cdots w_{n-1}$ is not divisible by any endpoints of $f_1$ or $f_2$. The condition $\partial w_i\cap \partial w_j=\emptyset$ if $i\neq j$ means that $\y$ is squarefree. Hence $(f_1\,\square\, f_2)\y$ is a Cartesian square contained in $\tok{G}{n}$. This gives a bijection between the 2-cells of $\ud{G}{n}$ and the 4-cycles of $\tok{G}{n}$ that agrees with the isomorphism $\tok{G}{n}\cong \sk_1 \ud{G}{n}$. Since $X(\tok{G}{n})$ is constructed by attaching a 2-cell to each 4-cycle, we obtain an isomorphism $X(\tok{G}{n})\cong \sk_2 \ud{G}{n}$.
\end{proof}
\label{prop:skel}
\end{prop}

\begin{proof}[Proof of Theorem \ref{thm:tokbraid}]
The theorem follows immediately from Propositions \ref{prop:dfg} and \ref{prop:skel}.
\end{proof}

\subsection{Local exchanges}

It is natural to ask what $A_1(\tok{G}{n})$ computes when $G$ contains 3- or 4-cycles. The answer is that $A_1(\tok{G}{n})$ is a quotient of $B_n(G)$, where the extra relations correspond to the 2-cells of $X(\tok{G}{n})$ not in $\ud{G}{n}$. By Proposition \ref{prop:dfg}, these 2-cells correspond to the 3-cycles and chordless 4-cycles of $\tok{G}{n}$ that are not Cartesian squares. We call these cycles of $\tok{G}{n}$ \emph{local exchanges}. In terms of the analogy with robots moving about a factory floor, local exchanges represent tasks performed by a small number of robots at a particular site. When $G$ contains 3- or 4-cycles, the group $A_1(\tok{G}{n})$ differs from $B_n(G)$ by essentially ignoring these tasks.

We illustrate three types of local exchanges in Figure \ref{fig:locex}. Consider, for example, Figure \ref{fig:locex}(a), where a single token is pictured moving around a 3-cycle of $G$. Assuming all other tokens remain stationary, this movement gives rise to a 3-cycle of $\tok{G}{n}$. Figures \ref{fig:locex}(b) and (c) depict configurations that give chordless 4-cycles of $\tok{G}{n}$. Each of these three types of local exchange has a \emph{complement}, in which the white vertices represent tokens and the black vertices are empty. For example, the complement of type (a) consists of 2 tokens moving around a 3-cycle of $G$. We denote these complementary types by (a'), (b') and (c'), respectively. These six types partition the set of local exchanges.

\begin{figure}[ht]
\centering
\begin{tikzpicture}[scale=0.9]
\def\a{3.3}
\draw (-\a/2,{sqrt(3)/2}) node {(a)};
\draw (0,0) -- (2,0) -- (1,{sqrt(3)})--cycle;
\draw (\a,0) -- (\a+2,0) -- (\a+1,{sqrt(3)})--cycle;
\draw (2*\a,0) -- (2*\a+2,0) -- (2*\a+1,{sqrt(3)})--cycle;
\draw (3*\a,0) -- (3*\a+2,0) -- (3*\a+1,{sqrt(3)})--cycle;
\foreach \x in {0,...,3} {
\draw[fill=white] (\x*\a,0) circle (0.12);
\draw[fill=white] (\x*\a+2,0) circle (0.12);
\draw[fill=white] (\x*\a+1,{sqrt(3)}) circle (0.12);
}
\draw[fill=black] (0,0) circle (0.12);
\draw[fill=black] (\a+2,0) circle (0.12);
\draw[fill=black] (2*\a+1,{sqrt(3)}) circle (0.12);
\draw[fill=black] (3*\a,0) circle (0.12);
\end{tikzpicture}

\vspace{5mm}
\begin{tikzpicture}[scale=0.8]
\def\a{3.5}
\draw (-\a/2,1) node {(b)};
\draw (0,0) -- (2,0) -- (2,2)--(0,2)--cycle;
\draw (\a,0) -- (\a+2,0) -- (\a+2,2)--(\a,2)--cycle;
\draw (2*\a+0,0) -- (2*\a+2,0) -- (2*\a+2,2)--(2*\a+0,2)--cycle;
\draw (3*\a+0,0) -- (3*\a+2,0) -- (3*\a+2,2)--(3*\a+0,2)--cycle;
\draw (4*\a+0,0) -- (4*\a+2,0) -- (4*\a+2,2)--(4*\a+0,2)--cycle;
\foreach \x in {0,...,4} {
\draw[fill=white] (\x*\a,0) circle (0.14);
\draw[fill=white] (\x*\a+2,0) circle (0.14);
\draw[fill=white] (\x*\a+2,2) circle (0.14);
\draw[fill=white] (\x*\a,2) circle (0.14);
}
\draw[fill=black] (0,0) circle (0.14);
\draw[fill=black] (\a+2,0) circle (0.14);
\draw[fill=black] (2*\a+2,2) circle (0.14);
\draw[fill=black] (3*\a,2) circle (0.14);
\draw[fill=black] (4*\a,0) circle (0.14);
\end{tikzpicture}

\vspace{5mm}
\begin{tikzpicture}[scale=0.8]
\def\a{3.5}
\draw (-\a/2,1) node {(c)};
\draw (0,0) -- (2,0) -- (2,2)--(0,2)--cycle;
\draw (\a,0) -- (\a+2,0) -- (\a+2,2)--(\a,2)--cycle;
\draw (2*\a+0,0) -- (2*\a+2,0) -- (2*\a+2,2)--(2*\a+0,2)--cycle;
\draw (3*\a+0,0) -- (3*\a+2,0) -- (3*\a+2,2)--(3*\a+0,2)--cycle;
\draw (4*\a+0,0) -- (4*\a+2,0) -- (4*\a+2,2)--(4*\a+0,2)--cycle;
\foreach \x in {0,...,4} {
\draw[fill=white] (\x*\a,0) circle (0.14);
\draw[fill=white] (\x*\a+2,0) circle (0.14);
\draw[fill=white] (\x*\a+2,2) circle (0.14);
\draw[fill=white] (\x*\a,2) circle (0.14);
}
\draw[fill=black] (0,0) circle (0.14);
\draw[fill=black] (2,2) circle (0.14);
\draw[fill=black] (\a+2,0) circle (0.14);
\draw[fill=black] (\a+2,2) circle (0.14);
\draw[fill=black] (2*\a+2,0) circle (0.14);
\draw[fill=black] (2*\a,2) circle (0.14);
\draw[fill=black] (3*\a,2) circle (0.14);
\draw[fill=black] (3*\a+2,2) circle (0.14);
\draw[fill=black] (4*\a,0) circle (0.14);
\draw[fill=black] (4*\a+2,2) circle (0.14);
\end{tikzpicture}
\caption{Three types of local exchanges.}
\label{fig:locex}
\end{figure}
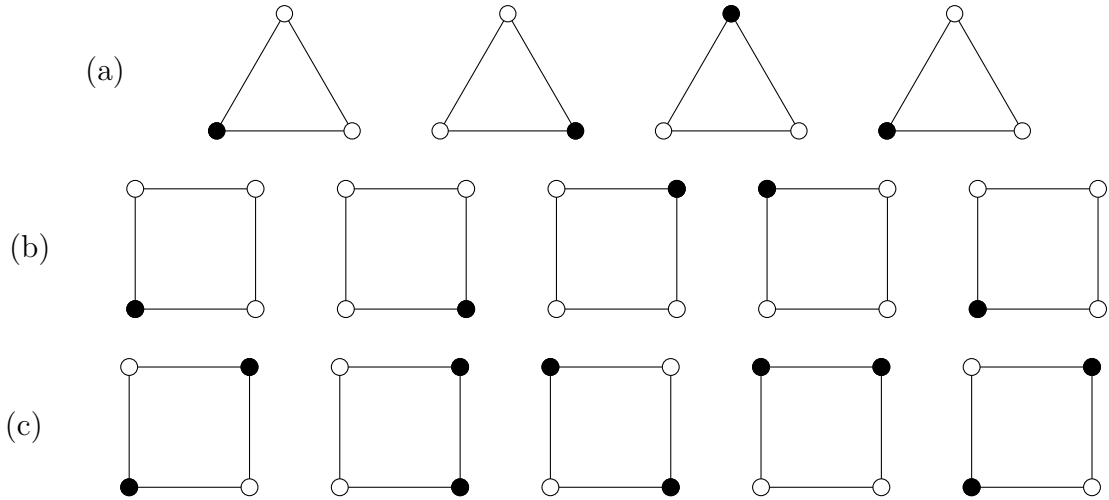

It is easy to count the local exchanges. Let $N$ denote the number of vertices of $G$. Assume that $n\geq 3$ and $N\geq n+3$. Each 3-cycle of $G$ contributes $\binom{N-3}{n-1}$ local exchanges of type (a) and $\binom{N-3}{n-2}$ of type (a'). Each chordless 4-cycle of $G$ contributes $\binom{N-4}{n-1}$ local exchanges of type (b), $\binom{N-4}{n-3}$ of type (b') and $2\binom{N-4}{n-2}$ each of types (c) and (c'). In total, there are
\[\kappa_3\left(\binom{N-3}{n-1}+\binom{N-3}{n-2}\right)+\kappa_4\left(\binom{N-4}{n-1}+4\binom{N-4}{n-2}+\binom{N-4}{n-3}\right)\]
local exchanges, where $\kappa_i$ denotes the number of chordless $i$-cycles of $G$. Hence $X(\tok{G}{n})$ is obtained from $\sk_2 \ud{G}{n}$ by attaching this many 2-cells. Adjustments to this count can be made in case $n<3$ or $N< n+3$.
\section{Open questions}
\label{sec:openq}

\subsection{Higher discrete homotopy groups of reduced powers}

Throughout this subsection, let $X$ be a simply-connected CW complex. There are a number of theorems relating the homotopy groups of $\sp{X}{n}$ to the singular homology groups of $X$. P. A. Smith proved the first such result, as mentioned in the introduction. A quarter-century later, Dold and Puppe \cite{dold1961} proved the following higher homotopy version of the theorem of Smith:

\begin{thm}[{\cite[Satz 12.11]{dold1961}}]
Suppose that $X$ is $k$-connected, i.e. that $\pi_i(X)$ is trivial for $i=1,\ldots, k$. If $n\geq 2$, then
\begin{equation*}
	\pi_i(\sp{X}{n})\cong \widetilde{H}_i(X)
\end{equation*}
for $i=0,\ldots,2n+k-1$, where $\widetilde{H}_i(X)$ is the $i$th reduced singular homology group of $X$.
\label{thm:doldpuppe}
\end{thm}

Currently, little is known about higher discrete homotopy groups (see, e.g., \cite{babson2006, lutz2020}). A positive answer to our next question would be a leap forward in discrete homotopy theory.

\begin{q}
Is there a discrete analog of Theorem \ref{thm:doldpuppe}? In other words, are there general conditions on $G$, $i$ and $n$ under which $A_i(\rp{G}{n}) \cong \widetilde{\hc}_i(G)$?
\end{q}

The most famous theorem on symmetric products was proven by Dold and Thom \cite{dold1956, dold1958}. Given a base point $x_0\in X$, we define (continuous) inclusions $\sp{X}{n}\to \sp{X}{n+1}$ by
\[[x_1,\ldots,x_n]\mapsto [x_0,x_1,\ldots,x_n].\]
This allows us to consider the infinite union
\[\sp{X}{\infty} = \bigcup_{n=1}^\infty \sp{X}{n},\]
where a set $C\subseteq \sp{X}{\infty}$ is closed if and only if $C\cap \sp{X}{n}$ is closed for all $n$. The space $\sp{X}{\infty}$ is called the \emph{infinite symmetric product} of $X$. While it depends on the base point $x_0$, this is typically suppressed in the notation.

\begin{thm}[{\cite{dold1956, dold1958}}]
If $X$ is a path-connected cell complex, then
\[\pi_i(\sp{X}{\infty}) \cong \widetilde{H}_i(X)\]
for all $i\geq 0$.
\label{thm:doldthom}
\end{thm}

We can convert these ideas into graphical terms. Given a base vertex $v_0\in G$, the function $\rp{G}{n}\to \rp{G}{n+1}$ defined in terms of monomials by
\begin{equation}
v_1\cdots v_n\mapsto v_0v_1\cdots v_n
\label{eq:grainc}
\end{equation}
is an injective graph map. Thus we can consider the infinite union
\[\rp{G}{\infty} = \bigcup_{n=1}^\infty \rp{G}{n}.\]
This (possibly infinite) graph is called the \emph{infinite symmetric product} of $G$. The following question was our original motivation for this paper.

\begin{q}
What can be said about the discrete homotopy groups $A_i(\rp{G}{\infty})$ with regard to the reduced discrete singular cubical homology groups $\widetilde{\hc}_i(G)$?
\end{q}

A satisfactory answer might be difficult to obtain for several reasons. First, it is difficult to obtain adequate descriptions of $\rp{G}{\infty}$ in all but the simplest examples. Second, the main object in the proof of Theorem \ref{thm:doldthom}, the notion of a \emph{quasifibration}, seems to have little meaning when translated directly into discrete terms. And third, we suspect that the groups $A_i(\rp{G}{\infty})$ are not finitely generated in general, and therefore probably not isomorphic to $\widetilde{\hc}_i(G)$. However, it would be interesting to know whether this leads to a new discrete homology theory, i.e. whether the functors given by $G\mapsto A_i(\rp{G}{\infty})$ satisfy discrete versions of the Eilenberg-Steenrod axioms (see e.g. \cite[Section 3]{barcelo2014}).

\subsection{Discrete Morse theory}

Discrete Morse theory gives a recipe to simplify any regular CW complex $X$ by reducing the number of cells while preserving the homotopy type. We provide a rough summary of this process; excellent introductions can be found in \cite{forman2002, kozlov2008}. The main ingredient is a \emph{discrete gradient vector field (DGVF)}, i.e. a particular type of partial matching on the face poset of $X$. The cells in this matching can be either removed or collapsed, resulting in a CW complex $Y$ that is homotopy equivalent to $X$. While the ability to eliminate cells is obviously favorable from an enumerative standpoint, this procedure can also reveal subtler topological data. For example, a well-chosen DGVF might yield a complex $Y$ with $\operatorname{dim} Y < \operatorname{dim} X$, so that $H_i(X)$ is trivial for all $i>\operatorname{dim} Y$.

There is a sizable literature on applications of discrete Morse theory to braid groups of graphs \cite{farley2005, farley2008, kim2012, ko2012}. Much of this work centers on a particular DGVF on the discrete configuration space $\ud{G}{n}$ defined by Farley and Sabalka \cite{farley2005}. This DGVF depends on choices of a spanning tree $T$ of $G$ and a depth-first search ordering of $T$. By making appropriate choices, one can obtain a simplification of $\ud{G}{n}$ with favorable properties. This line of inquiry leads, for example, to simple presentations of $B_n(G)$ and explicit formulas for $H_1(\uc{G}{n})$ in terms of graph invariants \cite{farley2012, ko2012}.

From Section 5, we know that if $G$ contains 3- or 4-cycles, then the CW complex $X(\tok{G}{n})$ is obtained from the 2-skeleton $\sk_2 \ud{G}{n}$ of the discrete configuration space by attaching 2-cells. To better understand the complex $X(\tok{G}{n})$, and hence the discrete fundamental group $A_1(\tok{G}{n})$, we are led to ask the following.

\begin{q}
Can the DGVF on $\ud{G}{n}$ be meaningfully adapted to $X(\tok{G}{n})$?
\end{q}

A meaningful adaptation would be one that eliminates as many of the extra attached 2-cells as possible, or one that leads to results resembling those on $B_n(G)$ and $H_1(\uc{G}{n})$. For example, it would be very interesting to find a graph-theoretic formula for the discrete singular cubical homology group $\hc_1(\tok{G}{n})$.
\section*{Acknowledgments}

The author thanks H{\'e}l{\`e}ne Barcelo and Curtis Greene for helpful comments.

\bibliographystyle{abbrv}
\bibliography{lutz-dhotc}

\begin{thebibliography}{10}

\bibitem{abrams2000}
A.~Abrams.
\newblock {\em Configuration Spaces and Braid Groups of Graphs}.
\newblock PhD thesis, University of California, Berkeley, 2000.

\bibitem{abrams2002}
A.~Abrams and R.~Ghrist.
\newblock Finding topology in a factory: Configuration spaces.
\newblock {\em Amer. Math. Monthly}, 109(2):140--150, 2002.

\bibitem{aguilar2008}
M.~Aguilar, S.~Gitler, and C.~Prieto.
\newblock {\em Algebraic Topology from a Homotopical Viewpoint}.
\newblock Universitext. Springer, 2008.

\bibitem{babson2006}
E.~Babson, H.~Barcelo, M.~de~Longueville, and R.~Laubenbacher.
\newblock Homotopy theory of graphs.
\newblock {\em J. Algebraic Combin.}, 24(1):31--44, 2006.

\bibitem{bantay2003}
P.~Bantay.
\newblock Permutation orbifolds and their applications.
\newblock In S.~Berman, Y.~Billig, Y.-Z. Huang, and J.~Lepowsky, editors, {\em
  Vertex Operator Algebras in Mathematics and Physics}, volume~39 of {\em
  Fields Inst. Commun.}, pages 13--23. Amer. Math. Soc., 2003.

\bibitem{barcelo2014}
H.~Barcelo, V.~Capraro, and J.~A. White.
\newblock {Discrete homology theory for metric spaces}.
\newblock {\em Bull. Lond. Math. Soc.}, 46(5):889--905, 2014.

\bibitem{barcelo2001}
H.~Barcelo, X.~Kramer, R.~Laubenbacher, and C.~Weaver.
\newblock Foundations of a connectivity theory for simplicial complexes.
\newblock {\em Adv. Appl. Math.}, 26(2):97--128, 2001.

\bibitem{barcelo2005}
H.~Barcelo and R.~Laubenbacher.
\newblock Perspectives on {$A$}-homotopy theory and its applications.
\newblock {\em Discrete Math.}, 298(1):39--61, 2005.

\bibitem{barcelo2011}
H.~Barcelo, C.~Severs, and J.~A. White.
\newblock {$k$}-parabolic subspace arrangements.
\newblock {\em Trans. Amer. Math. Soc.}, 363(11):6063--6083, 2011.

\bibitem{birman1969}
J.~S. Birman.
\newblock On braid groups.
\newblock {\em Comm. Pure Appl. Math.}, 22(1):41--72, 1969.

\bibitem{birman1974}
J.~S. Birman.
\newblock {\em Braids, Links and Mapping Class Groups}.
\newblock Ann. of Math. Stud. Princeton University Press, 1974.

\bibitem{crisp2004}
J.~Crisp and B.~Wiest.
\newblock Embeddings of graph braid and surface groups in right-angled {A}rtin
  groups and braid groups.
\newblock {\em Algebr. Geom. Topol.}, 4:439--472, 2004.

\bibitem{delabie2020}
T.~Delabie and A.~Khukhro.
\newblock Coarse fundamental groups and box spaces.
\newblock {\em Proc. Roy. Soc. Edinburgh Sect. A}, 150(3):1139--1154, 2020.

\bibitem{dold1961}
A.~Dold and D.~Puppe.
\newblock Homologie nicht-additiver {F}unktoren. {A}nwendungen.
\newblock {\em Ann. Inst. Fourier}, 11:201--312, 1961.

\bibitem{dold1956}
A.~Dold and R.~Thom.
\newblock Une g\'{e}n\'{e}ralisation de la notion d'espace fibr\'{e}:
  {A}pplication aux produits sym\'{e}triques infinis.
\newblock {\em C. R. Math. Acad. Sci.}, 242:1680--1682, 1956.

\bibitem{dold1958}
A.~Dold and R.~Thom.
\newblock Quasifaserungen und {U}nendliche {S}ymmetrische {P}rodukte.
\newblock {\em Ann. of Math.}, 67(2):239--281, 1958.

\bibitem{fabmon2012}
R.~Fabila-Monroy, D.~Flores-Pe\~{n}aloza, C.~Huemer, F.~Hurtado, J.~Urrutia,
  and D.~R. Wood.
\newblock Token graphs.
\newblock {\em Graphs Combin.}, 28(3):365--380, 2012.

\bibitem{farber2003}
M.~Farber.
\newblock Topological complexity of motion planning.
\newblock {\em Discrete Comput. Geom.}, 29:211--221, 2003.

\bibitem{farley2005}
D.~Farley and L.~Sabalka.
\newblock Discrete {M}orse theory and graph braid groups.
\newblock {\em Algebr. Geom. Topol.}, 5:1075--1109, 2005.

\bibitem{farley2008}
D.~Farley and L.~Sabalka.
\newblock On the cohomology rings of tree braid groups.
\newblock {\em J. Pure Appl. Algebra}, 212(1):53--71, 2008.

\bibitem{farley2012}
D.~Farley and L.~Sabalka.
\newblock Presentations of graph braid groups.
\newblock {\em Forum Math.}, 24(4), 2012.

\bibitem{forman2002}
R.~Forman.
\newblock A user's guide to discrete {M}orse theory.
\newblock {\em S\'{e}m. Loth. Combin.}, 48:1--35, 2002.

\bibitem{ghrist2001}
R.~Ghrist.
\newblock Configuration spaces and braid groups on graphs in robotics.
\newblock In J.~Gilman, W.~W. Menasco, and X.-S. Lin, editors, {\em Knots,
  Braids, and Mapping Class Groups---Papers Dedicated to Joan S. Birman},
  volume~24 of {\em AMS/IP Stud. Adv. Math.} Amer. Math. Soc., 2001.

\bibitem{ghrist2007}
R.~Ghrist and V.~Peterson.
\newblock The geometry and topology of reconfiguration.
\newblock {\em Adv. Appl. Math.}, 38(3):302--323, 2007.

\bibitem{hammack2016}
R.~H. Hammack and G.~D. Smith.
\newblock Cycle bases of reduced powers of graphs.
\newblock {\em Ars Math. Contemp.}, 12(1), 2016.

\bibitem{kallel2016}
S.~Kallel and I.~Saihi.
\newblock Homotopy groups of diagonal complements.
\newblock {\em Algebr. Geom. Topol.}, 16(5):2949--2980, 2016.

\bibitem{kim2012}
J.~H. Kim, K.~H. Ko, and H.~W. Park.
\newblock Graph braid groups and right-angled {A}rtin groups.
\newblock {\em Trans. Amer. Math. Soc.}, 364(1):309--360, 2012.

\bibitem{ko2012}
K.~H. Ko and H.~W. Park.
\newblock Characteristics of graph braid groups.
\newblock {\em Discrete Comput. Geom.}, 48(4):915--963, 2012.

\bibitem{kozlov2008}
D.~Kozlov.
\newblock {\em Combinatorial Algebraic Topology}, volume~21 of {\em Algorithms
  Comput. Math.}
\newblock Springer, 2008.

\bibitem{lutz2020}
B.~Lutz.
\newblock Higher discrete homotopy groups of graphs.
\newblock {\em ar{X}iv e-prints}, 2020.
\newblock arXiv:2003.02390.

\bibitem{macdonald1962}
I.~G. Macdonald.
\newblock Symmetric products of an algebraic curve.
\newblock {\em Topology}, 1(4):319--343, 1962.

\bibitem{prue2014}
P.~Prue and T.~Scrimshaw.
\newblock Abrams's stable equivalence for graph braid groups.
\newblock {\em Topology Appl.}, 178:136--145, 2014.

\bibitem{smith1936}
P.~A. Smith.
\newblock Manifolds with abelian fundamental groups.
\newblock {\em Ann. of Math.}, 37(3):526--533, 1936.

\bibitem{sagemath}
{The Sage Developers}.
\newblock {\em {S}ageMath, the {S}age {M}athematics {S}oftware {S}ystem
  ({V}ersion 9.0)}, 2020.
\newblock {\tt https://www.sagemath.org}.

\end{thebibliography}
\end{document}